\documentclass[reqno, a4paper]{amsart}

\input xy
\xyoption{all}\xyoption{poly}
\newdir{ >}{{}*!/-9pt/@{>}}

\usepackage[all]{xy}
\newcommand{\catC}{\mathcal{C}}
\newcommand{\catH}{\mathcal{H}}
\newcommand{\Qnd}{\mathsf{Qnd}}
\newcommand{\Ker}{\operatorname{Ker}}

\newcommand{\Qndt}{\mathsf{Qnd}^{\ast}}
\newcommand{\Inn}{\operatorname{Inn}}

\newcommand{\Ext}{\operatorname{Ext}}
\newcommand{\TExt}{\operatorname{TExt}}
\newcommand{\CExt}{\operatorname{CExt}}
\newcommand{\NExt}{\operatorname{NExt}}

\newcommand{\Grp}{\mathsf{Grp}}

\newcommand{\Ab}{\mathsf{Ab}}
\newcommand{\ab}{\mathsf{ab}}
\newcommand{\Aut}{\operatorname{Aut}}
\newcommand{\Eq}{\operatorname{Eq}}
\newcommand{\id}{\operatorname{id}}
\newcommand{\lhdi}{\lhd^{-1}}
\newcommand{\catA}{\mathcal{A}}
\newcommand{\catX}{\mathcal{X}}

\renewcommand{\Im}{\operatorname{Im}}
\usepackage{graphicx}
\usepackage{color}

\newdir{ >}{{}*!/-10pt/\dir{>}}
\def\cartesien{%
    \ar@{-}[]+R+<6pt,-1pt>;[]+RD+<6pt,-6pt>%
    \ar@{-}[]+D+<1pt,-6pt>;[]+RD+<6pt,-6pt>%
}
\usepackage{todonotes}
\usepackage[colorlinks]{hyperref}
\hypersetup{%
    urlcolor=blue,linkcolor=blue,citecolor=blue
}
\usepackage[T1]{fontenc}
\usepackage[utf8]{inputenc}

\usepackage{lineno,hyperref}
\modulolinenumbers[5]

\usepackage{tikz}
\usetikzlibrary{decorations.pathreplacing,matrix,positioning,decorations.pathmorphing,knots,calc}
\usepackage{tikz-cd}
\usepackage{textcomp}
\usepackage{graphicx}
\usepackage{hyperref}

\usepackage{todonotes}
\usepackage{graphicx}
\newdir{ >}{{}*!/-10pt/\dir{>}}
\def\cartesien{%
    \ar@{-}[]+R+<6pt,-1pt>;[]+RD+<6pt,-6pt>%
    \ar@{-}[]+D+<1pt,-6pt>;[]+RD+<6pt,-6pt>%
  }
\newdir{-|>}{!/ 9pt / @{|}*!/ 4.5pt /  :(1,-.2)@^{>}*!/ 4.5pt /  :(1,+.2)@_{>}*@{ }*@{ }*@{ } }

\newtheorem{theorem}{Theorem}[section]
\newtheorem{lemma}[theorem]{Lemma}
\newtheorem{proposition}[theorem]{Proposition}
\newtheorem{corollary}[theorem]{Corollary}
\newtheorem{definition}[theorem]{Definition}
\newtheorem{example}[theorem]{Example}

\begin{document}

\author{M.~Duckerts-Antoine}
\address{Centre for Mathematics, University of Coimbra, Department of Mathematics, Apartado 3008, 3001-501 Coimbra, Portugal}
\email{mathieud@mat.uc.pt}

\author{V.~Even}
\address{Institut de Recherche en Math\'ematique et Physique,  Universit\'e catholique de Louvain, Chemin du Cyclotron 2, 1348 Louvain-la-Neuve, Belgium.}
\email{valerian.even@uclouvain.be}

\author{A.~Montoli}
\address{Dipartimento di Matematica ``Federigo Enriques'', Universit\`{a} degli Studi di Milano, Via Saldini 50, 20133 Milano, Italy.}
\email{andrea.montoli@unimi.it}

\title[How to centralize and normalize quandle extensions]{How to centralize and normalize\\ quandle extensions}

\begin{abstract} We show that quandle coverings in the sense of Eisermann form a
(regular epi) reflective subcategory of the category of surjective
quandle homomorphisms, both by using arguments coming from
categorical Galois theory and by constructing concretely a
centralization congruence. Moreover, we show that a similar result
holds for normal quandle extensions.
\end{abstract}

\maketitle


\section{Introduction}


The origins of the structure of quandle go back to the early
$40$'s when M.~Takasaki \cite{Tak43} defined the notion of a kei
in order to find an algebraic structure to capture the properties
of reflections in a Euclidean space. A \emph{kei} is defined as a
set~$A$ equipped with a binary operation $\lhd$ satisfying the
following three identities for all $a,\ b,\ c \in A$ :
\begin{itemize}
    \item $a\lhd a = a$;
    \item $(a \lhd b) \lhd b = a$;
    \item $(a \lhd b) \lhd c = (a \lhd c) \lhd (b \lhd c)$.
\end{itemize}
The notation $a \lhd b$ stands for the reflection of $a$ over $b$.

Forty years later arose the structure of quandle, as defined by D.
Joyce~\cite{Joy82a}. The aim of this structure was to construe the
symmetries of a geometric object on the object itself. In
particular, quandles have interesting interactions with knot
theory where they actually provide a knot invariant. In fact,
replacing the knot group with the knot quandle, it was proved
in~\cite{Joy82a} that two tame knots with isomorphic knot quandles
are equivalent up to orientation.

\begin{definition}
A \emph{quandle} is a set $A$ equipped with two binary operations
$\lhd$ and~$\lhdi$ such that for all $a,\ b,\ c \in A$:
\begin{enumerate}
    \item[(Q1)] $a \lhd a = a$ \hfill (idempotency);
    \item[(Q2)] $(a \lhd b) \lhdi b = a = (a \lhdi b) \lhd b$  \hfill (right invertibility);
    \item[(Q3)] $(a \lhd b) \lhd c = (a \lhd c) \lhd (b \lhd c)$ \hfill  (self-distributivity).
\end{enumerate}
\end{definition}

Given two quandles $A$ and $B$, a function $f \colon A \to B$ is a
\emph{quandle homomorphism} when it preserves the two binary
operations: the equalities $f(a \lhd a') = f(a) \lhd f(a')$ and
$f(a \lhdi a') = f(a) \lhdi f(a')$ hold for all $a, a' \in A$. We
denote the category of quandles by $\Qnd$ and call a surjective
quandle homomorphism an \emph{extension}.

In this article, we will be mainly interested in two classes of
extensions in $\Qnd$: the classes of central and normal
extensions. Both notions come from categorical Galois theory and
they are defined with respect to a fundamental adjunction linking
the category of quandles to its full subcategory of trivial
quandles (a quandle $A$ is trivial when $a \lhd a' = a$ and $a
\lhdi a' = a$ for all $a, a'$ in $A$).
\begin{equation}\label{adjqndt}
\begin{tikzcd}
\Qnd \ar[rr,bend left,"\pi_0"] & \perp & \Qndt \ar[ll,bend
left,"\supseteq"]
\end{tikzcd}
\end{equation}

Note that the notion of central extension that we study here
corresponds to another notion (specific to the context of
quandles) introduced by M.~Eisermann. In~\cite{Eis14}, he
developed a Galois theory for what he called \emph{quandle
coverings}.  A quandle homomorphism $f \colon A \to B$ is a
quandle covering if it is surjective and $f(a) = f(a')$ implies $c
\lhd  a = c \lhd a'$ for all $c \in A$. The coincidence of the two
notions was proved in~\cite{Eve14a}.

On the contrary, normal extensions do not seem to have been
studied earlier in the context of quandles. Nevertheless, they can
be also described easily:  a surjective homomorphism $f \colon A
\to B$ is a normal extension if, for all $a_i \in A$, $\alpha_i
\in \{ -1,1\}$ with $0\leq i \leq n$,
\[
a_0 \lhd^{\alpha_1} a_1 \lhd^{\alpha_2} \dots \lhd^{\alpha_n} a_n
= a_0,
\]
implies
\[
a'_0 \lhd^{\alpha_1} a'_1 \lhd^{\alpha_2} \dots \lhd^{\alpha_n}
a'_n = a'_0
\]
for all $a'_i \in f^{-1}(f(a_i))$.

The aim of this article is to prove that it is possible to
centralize and normalize any quandle extension. More formally, we
shall prove that there are left adjoints to the inclusion functors
$\CExt(\Qnd) \hookrightarrow \Ext(\Qnd)$ and $\NExt(\Qnd)
\hookrightarrow \Ext(\Qnd)$ where $\CExt(\Qnd)$ is the category of
central extensions, $\NExt(\Qnd)$ is the category of normal
extensions, and $\Ext(\Qnd)$ the category of extensions (all
viewed as full subcategories of the category $\Qnd^\rightarrow$ of
arrows in $\Qnd$). The construction of a universal central (or
normal) extension associated with any surjective quandle
homomorphism can be used to obtain a description of the
fundamental group of a quandle, and to relate it with cohomology
and extension theory of quandles. It is, moreover, the first step
in order to get a homotopy theory for quandles (this is material
for future work).

\section{On the category of quandles}

We already recalled in the Introduction the notion of quandle.
Here we provide some basic examples of quandles:

\begin{example}[\cite{Joy82a}]
    \begin{enumerate}
        \item Any set $A$ is a quandle with the structure defined by \[a \lhd a' = a = a \lhdi a'\] for
        all $a,\ a' \in A$. Such quandles are called \emph{trivial quandles}.
        \item  Given a multiplicative group $G$, we can define \[g \lhd h = h^{-1} \cdot g \cdot h\] and
         \[g \lhdi h = h \cdot  g \cdot h^{-1}\] for all $g,\ h \in G$. The group $G$ equipped with these operations is
         a quandle called the \emph{conjugation quandle}. This is,
         in some sense, the key example of quandles.
         Indeed, as observed in \cite{Joy82a}, the Wirtinger
         presentation of the knot group of a given knot only
         involves conjugations. Whence the idea of replacing knot
         groups with knot quandles: the axioms defining a quandle
         have been obtained from those satisfied by the group
         conjugation. As we recalled in the Introduction, the knot
         quandle of a tame knot characterizes the knot up to
         orientation.
        \item Let $n$ be a positive integer and define on $\mathbb{Z}_n$ the operations
        \[i \lhd j = 2j - i = i \lhdi j\] for all $i$, $j \in \mathbb{Z}_n$. This defines a quandle called \emph{dihedral quandle}.
        \item More generally, if $G$ is a multiplicative group, we can define \[g \lhd h = h \cdot g^{-1}
        \cdot h = g \lhdi h\] for all $g,\ h \in G$. This defines a quandle called \emph{core quandle}.
        \item Let $M$ be a module over the ring $\mathbb{Z}[t,t^{-1}]$ of Laurent polynomials. Define
        \[x \lhd y = t(x-y)+y\] and \[x \lhd^{-1} y = t^{-1}(x-y)+y\] for all $x,\ y \in M$. This defines a quandle
        called \emph{Alexander quandle}. These quandles can be used to compute the Alexander polynomial of a knot.
    \end{enumerate}
\end{example}

A well-known fact about any quandle $A$ is that the axioms (Q2)
and (Q3) imply that, for every $b \in A$, the \emph{right
translation} $(-)^{\rho_b} \colon A \to A$ defined by $a^{\rho_b}
= a \lhd b$  is an  automorphism.

\begin{definition}
    The group $\Inn(A)$ of \emph{inner automorphisms} of a quandle $A$ is the subgroup of $\Aut(A)$ (the group of all automorphisms
    of $A$) generated by all right translations $\rho_b$ with $b \in A$.
\end{definition}

From a quandle $A$, we define a connected component to be an orbit
under the action of the group $\Inn(A)$ (for $a\in A$, we write
$\eta_A(a)$ for its orbit). The set of connected components of $A$
is denoted by $\pi_0(A)$ and yields a trivial quandle. The
reflection of the category of quandles into its full subcategory
of trivial quandles is precisely given by $\pi_0 \colon \Qnd \to
\Qndt$, which takes a quandle $A$ and send it to its set of
connected components $\pi_0(A)$. The component of the unit $\eta$
at $A$ is given by $\eta_A \colon A \to \pi_0(A) \colon a \mapsto
\eta_A(a)$.

We recall that a \emph{variety}  (in the sense of universal
algebra) is a class of algebras of signature $\mathcal{F}$
satisfying a set of identities of the same signature
$\mathcal{F}$, with morphisms preserving the operations in
$\mathcal{F}$ (see Birkhoff's theorem~\cite{Bir35}).

Like any other variety, $\Qnd$ is in particular a regular category
\cite{Barr}. This means that
\begin{enumerate}
\item every morphism $f$ can be factored as $f = m\circ p$ where
$m$ is a monomorphism and $p$ is a \emph{regular epimorphism} (the
coequaliser of some pair of parallel morphisms), and such a
factorization is unique up to isomorphisms; \item every pullback
of a regular epimorphism along any morphism is a regular
epimorphism (we say that regular epimorphisms are
\emph{pullback-stable}).
\end{enumerate}
We recall that in $\Qnd$, as in any variety, a monomorphism is an
injective homomorphism and a regular epimorphism is a surjective
homomorphism. A consequence of the above axioms is that regular
epimorphisms are \emph{orthogonal} to monomorphisms, by which we
mean that for every commutative diagram
\[
\begin{tikzcd}
    A \ar[r,"p"] \ar[d,"u"'] & B \ar[d,"v"] \\
    C \ar[r,"m"'] & D
\end{tikzcd}
\]
where $p$ is a regular epimorphism and $m$ a monomorphism, there
exists a unique morphism $t\colon B \to C$ such that $u=t\circ p$
and $v = m \circ t$. In particular, this implies that a
factorization $f =m\circ p$ as above is unique (up to
isomorphism). We shall use this fact in Section~\ref{Normalizing
an extension}.

\section{Central and normal extensions} \label{section central and normal extensions}

The adjunction~\eqref{adjqndt} is at the heart of this paper. The
reason is that it fits in the categorical Galois theory developed
by G. Janelidze~\cite{Jan90}. In particular, the categorical
Galois theory studies central extensions and normal extensions
defined with respect to adjunctions that satisfy a certain
pullback preservation property. Such adjunctions are said to be
admissible (see the beginning of Section~\ref{Trivializing an
extension} for the definition). In order to give a hint of this
theory, let us restrict ourselves to the case of an adjunction
\begin{equation}\label{A}
\begin{tikzcd}
\catC \arrow[rr,bend left, "I"] & \perp & \catH \arrow[ll,bend
left, "H"]
\end{tikzcd}
\end{equation}
where $\catC$ is a variety and $\catH$ a \emph{subvariety} of
$\catC$ (see \cite{JK94} for a more general approach). This
basically says that the components of the unit (resp.~counit) are
regular epimorphisms (resp.~isomorphisms), $H$ is an inclusion,
and $\catH$ is closed under quotients (if $f\colon A \to B$ is a
surjective homomorphism and $A$ is in $\catH$, $B$ is also in
$\catH$). In particular, the adjunction is  a \emph{(regular
epi)-reflection}.

Relatively to any adjunction~\eqref{A}, it is possible to define
the notions of trivial extension, central extension, and normal
extension. These are all special surjective homomorphisms in the
category $\catC$. A surjective homomorphism $f \colon A \to B$
will often be called an \emph{extension (of $B$)}.

\begin{definition}
    An extension $f \colon A \to B$ is
\begin{itemize}
\item \emph{trivial} when the naturality square of the adjunction
    \begin{equation}\label{nat}
    \begin{tikzcd}
    A \ar[r,"\eta_A"] \ar[d,"f"'] & HI(A) \ar[d,"HI(f)"] \\
    B \ar[r,"\eta_B"'] & HI(B)
    \end{tikzcd}
    \end{equation}
    is a pullback;
\item \emph{central} when there exists a surjective homomorphism
$p \colon E \to B$ such that $p_1 \colon E \times_B A \to E$ in
the pullback
    \begin{equation}
    \begin{tikzcd}
    E \times_B A \ar[r,"p_2"] \ar[d,"p_1"'] & A \ar[d,"f"] \\
    E \ar[r,"p"'] & B
    \end{tikzcd}
    \end{equation}
 is a trivial extension (in that case we say that $f$ is \emph{split by} $p$);
\item \emph{normal} when $f_1 \colon \Eq(f) \to A$ in the pullback
    \begin{equation}
    \begin{tikzcd}
    \Eq(f) \ar[r,"f_2"] \ar[d,"f_1"'] & A \ar[d,"f"] \\
    A \ar[r,"f"'] & B
    \end{tikzcd}
    \end{equation}
     is a trivial extension (i.e.~$f$ is split by itself). Here
     $\Eq(f)$ is the kernel congruence of $A$: two elements $a_1,
     a_2 \in A$ are in relation w.r.t. $\Eq(f)$ if and only if
     $f(a_1)  = f(a_2)$.
\end{itemize}
\end{definition}

As proved in \cite{Eve14a}, for our adjunction~\eqref{adjqndt}, an
extension $f \colon A \to B$ is trivial if and only if the
following condition $(T)$ holds:
\begin{equation}\label{T}
\forall \, a,\ a' \in A, \text{ if } f(a) = f(a') \text{ and }
\pi_0(a) = \pi_0(a'), \text{ then } a = a'. \tag{T}
\end{equation}

As explained before, the central extensions relative to the
adjunction~\eqref{adjqndt} turn out to be exactly the
\emph{quandle coverings} defined by M. Eisermann~\cite{Eis14}: $f
\colon A \to B$ is a quandle covering if it is surjective and
$f(a) = f(a')$ implies $c \lhd  a = c \lhd a'$ for all $c \in A$.
The proof of this result can be found in~\cite{Eve14a}. It is
interesting for us to note that this proof uses the existence of a
special central extension $p \colon \tilde{A} \to A$ for any
quandle $A$ which is ``weakly universal'', i.e. such that every
central extension of $A$ is actually split by $p$
(See~\cite[Theorem 2]{Eve14a} and its proof). In particular $p
\colon \tilde{A} \to A$ is a normal extension. We shall make use
of those special normal extensions in Section~\ref{Centralizing an
extension}.

In many cases, normal extensions coincide with central extensions.
A typical example is given by the (regular epi)-reflection between
the variety of groups and its subvariety of abelian groups
\begin{equation}\label{adjgrpab}
\begin{tikzcd}
\Grp \ar[rr,"\ab",bend left] & \perp &  \Ab \ar[ll,bend
left,"\supseteq"]
\end{tikzcd}
\end{equation}
Here $\ab$ is the abelianization functor sending a group $G$ to
$G/[G,G]$. The central extensions for this adjunction are exactly
the classical central extensions: a group homomorphism $f\colon
A\to B$ is central if and only if its kernel $\Ker(f)$ is a
subgroup of the centre $Z(A)$ of $A$. For our
adjunction~\eqref{adjqndt}, it is not true that normal and central
extensions coincide. In the following proposition we give a
description of the normal extensions and, via an example (see
Example~\ref{Example central is not normal} below), we show that
the central extensions are not always normal extensions (of
course, the converse is true: every normal extension is central).

\begin{proposition}\label{normal}
    A surjective quandle homomorphism $f \colon A \to B$ is a normal extension if and only if the following condition $(N)$ holds:

    for all $a_i \in A$ (with $0\leq i \leq n$) and $\alpha_j \in \{-1,1\}$  (with $1\leq j \leq n$),
    if \[a_0 \lhd^{\alpha_1} a_1 \lhd^{\alpha_2} \dots \lhd^{\alpha_n} a_n = a_0,\] then
    \[a'_0 \lhd^{\alpha_1} a'_1 \lhd^{\alpha_2} \dots \lhd^{\alpha_n} a'_n = a'_0\] for all $a'_i \in f^{-1}(f(a_i))$.
\end{proposition}

\begin{proof}
    By definition, $f \colon A \to B$ is a normal extension if and only if the first
    projection $f_1 \colon \Eq(f) \to~A$ in the following diagram is a trivial extension.
    \[\begin{tikzcd}
    \Eq(f) \ar[r,"f_2"] \ar[d,"f_1"'] & A \ar[d,"f"] \\
    A \ar[r,"f"'] & B
    \end{tikzcd}\]
    But $f_1 \colon \Eq(f) \to A$ is a trivial extension if and only if Condition $(T)$ holds:

    $\forall \, (a_0,a_0'), (x_0,x'_0) \in \Eq(f)$, if \[f_1(a_0,a'_0) = f_1(x_0,x'_0)\] (i.e. $a_0=x_0$) and
    \[\pi_0((a_0,a'_0)) = \pi_0((x_0,x'_0)), \] then \[(a_0,a'_0)=(x_0,x'_0).\]
    This translates to the following condition:

    $ \forall \, a_0,\ a'_0,\ x'_0$ such that $f(a_0)=f(a'_0)=f(x'_0)$, if there exists $(a_i,a'_i) \in \Eq(f)$ with
    $1 \leq i \leq n$ such that \[(a_0,a'_0) \lhd^{\alpha_1} (a_1,a'_1) \lhd^{\alpha_2} \dots \lhd^{\alpha_n} (a_n,a'_n) = (a_0,x'_0)\]
    then $(a_0,a'_0) = (a_0,x'_0)$, which means \[(a_0,a'_0) = (a_0,a'_0) \lhd^{\alpha_1} (a_1,a'_1) \lhd^{\alpha_2} \dots \lhd^{\alpha_n} (a_n,a'_n). \]
    Clearly, condition $(N)$ implies the previous condition, but it is also true that the previous condition
    implies $(N)$ since we can take $x'_0$ to be \[x'_0 = a'_0 \lhd^{\alpha_1} a'_1 \lhd^{\alpha_2} \dots \lhd^{\alpha_n} a'_n. \]
\end{proof}

    \begin{example}\label{Example central is not normal}
        Consider the involutive ($\lhd = \lhdi$) quandle $A$ given by
        \[\begin{tabular}{c|cccc}
        $\lhd$ & $a$ & $b$ & $c$ & $d$\\
        \hline
        $a$ & $a$ & $a$ & $a$ & $a$\\
        $b$ & $b$ & $b$ & $d$ & $b$\\
        $c$ & $c$ & $c$ & $c$ & $c$\\
        $d$ & $d$ & $d$ & $b$ & $d$
        \end{tabular}\]
        and the two-elements trivial quandle $X =\{x,y\}$. Now consider $f\colon A \to X$ defined by
        $f(a)=f(b)=f(d)=x$ and $f(c)=y$. It is not a normal extension since $f(a)=f(b)$ and $a \lhd c =a$ but
        $b \lhd c = d \neq b$. To see that it satisfies Condition $(C)$, it suffices to observe that elements with the same
        image by $f$ act in the same way, when we compute $\lhd$ with them on the right or, in other terms, when they give the same column in the
        composition table above.
    \end{example}

    Let us fix some notations. It is standard to write $\Qnd^\rightarrow$ for the category of arrows
    in $\Qnd$. The objects of $\Qnd^\rightarrow$ are the quandle homomorphisms and the arrows of $\Qnd^\rightarrow$ are the commutative
    squares in $\Qnd$. More formally, if $f\colon A\to B$ and $g\colon C\to D$ are two objects of $\Qnd^\rightarrow$, a morphism
    from $f$ to $g$ in $\Qnd^\rightarrow$ is a pair $(\alpha_1 \colon A \to C,\alpha_0\colon B \to D)$ such that
    $\alpha_0 \circ f =g\circ \alpha_1$, i.e.~the square
\[
\begin{tikzcd}
A \ar[r,"\alpha_1"] \ar[d,"f"'] & C \ar[d,"g"] \\
B \ar[r,"\alpha_0"'] & D
\end{tikzcd}
\]
commutes.

    We write $\Ext(\Qnd)$ for the full subcategory of $\Qnd^\rightarrow$ whose objects are the extensions.
    Similarly, $\CExt(\Qnd)$ (resp.~$\NExt(\Qnd)$) denotes the full subcategory of  $\Qnd^\rightarrow$ whose objects are
    central (resp.~normal) extensions. The category $\Ext(B)$ is the full subcategory of the comma category $(\Qnd \downarrow B)$
    determined by the extensions of $B$ (so that a morphism in $\Ext(B)$ is a commutative triangle). We will write $\TExt(B)$, $\CExt(B)$ and
    $\NExt(B)$ for the full subcategories of $\Ext(B)$ determined by the trivial extensions of $B$, central extensions of $B$, and normal
    extensions of $B$, respectively.

 The following inclusions are always true:
\[
\TExt(B) \subseteq \NExt(B) \subseteq \CExt(B) \subseteq \Ext(B).
\]
Note that the first inclusion above comes from the admissibility
of the adjunction. Indeed, for admissible adjunctions, the trivial
extensions are pullback-stable \cite[Proposition 4.1]{JK94}. For
this reason, the central and normal extensions also enjoy pullback
stability \cite[Proposition 4.3]{JK94}.

We now make some remarks on our two main problems.

\begin{enumerate}
\item To find a left adjoint to the inclusion functor $\CExt(\Qnd)
\hookrightarrow \Ext(\Qnd)$. We are looking here for a way to
transform universally any extension into a central extension. Due
to the pullback stability of central extensions, we can split this
problem into smaller problems: for each quandle $B$, find a left
adjoint to  the inclusion $\CExt(B) \hookrightarrow \Ext(B)$. In
other words, if $f\colon A \to B$ is an extension, we are
searching for a decomposition \[\begin{tikzcd}
A \arrow[rr, "f"] \arrow[rd, "q"'] & & B\\
& \overline{A} \arrow[ur,"c_f"'] &
    \end{tikzcd}\] where $c_f$ is central and  it is universal: if $f = c' \circ q'$ with $c' : A' \to B$ a central extension
    of $B$, there must exist a unique $\phi \colon \overline{A} \to A'$ such that $c'\circ \phi = c_f$ and $\phi \circ q = q'$. We will
    give two methods of doing so, one of which uses a categorical approach while the other uses an algebraic approach. It turns out that
    $q$ is always a regular epimorphism, so that our problem further reduces to finding an appropriate congruence $R_c$ on $A$.
\item To find a left adjoint to the inclusion functor $\NExt(\Qnd)
\hookrightarrow \Ext(\Qnd)$. Although, in principle, the same
procedure as above could be followed, we were not able to  tackle
the problem in that way. Instead, we use a general argument (the
Freyd Adjoint Functor Theorem) in order to solve the problem. Note
that this method can also be used to solve the first problem but
its drawback is that it is not really constructive: we are left
with no good information on how to construct the left adjoint we
are looking for, only its existence is proved.
\end{enumerate}


\section{Trivializing an extension}\label{Trivializing an extension}


 The adjunction \eqref{A} is said to be \emph{admissible} when the left adjoint functor $I \colon \catC \to \catH$ preserves pullbacks of the following form,
 where $H(f)$ is a surjective homomorphism:
 \begin{equation}\label{admipb}
 \begin{tikzcd}[row sep=3em]
 B \times_{HI(B)} H(X) \ar[r,"p_2"] \ar[d,"p_1"'] & H(X) \ar[d,"H(f)"] \\
 B \ar[r,"\eta_B"'] & HI(B).
 \end{tikzcd}
 \end{equation}
A famous example of an admissible adjunction is the
reflection~\eqref{adjgrpab} of the category $\Grp$ into $\Ab$. As
explained in \cite{JK94}, this fact is a consequence of an
important property of the variety $\Grp$ of groups: it is a
\emph{Mal'tsev variety}, i.e. a variety in which congruences over
any object \emph{permute} in the sense of the composition of
relations. Here we adopt the classical terminology used in
universal algebra and we call an equivalence relation $R \subseteq
A \times A$ on (the underlying set of) a quandle A, a
\emph{congruence}, if it has the property that $R$ is also a
subquandle of the product quandle $(A \times A,\lhd,\lhdi)$, so
that, for any $(a,b) \in R$ and $(a',b') \in R$, both the elements
\[
 (a, b) \lhd (a', b') = (a \lhd a', b \lhd b')
 \]
 and
 \[
 (a,b) \lhdi (a',b')=(a \lhdi a',b \lhdi b')
 \]
belong to the relation $R$.

    Given two congruences $R$ and $S$ on a quandle $A$, their (relational) composite $S \circ R$ is defined as the following relation on $A$:
    $$S \circ R = \{ (a,b) \in A \times A \, \mid \exists \, c \in A \  \mathrm{with}  \, (a,c) \in R\, \, \mathrm{and } \, \, (c, b) \in S \}.$$
    For a variety $\catC$ of algebras, having permutable congruences, meaning that we have $R \circ S = S \circ R$ for all congruences
    $R$ and $S$ on any object $C \in \catC$, is equivalent to the fact that the corresponding theory has a ternary term $p(a,b,c)$ such that
    $p(a,b,b)=a$ and $p(a,a,b)=b$ (see~\cite{Mal54}). For the variety of groups, such a ternary term is given by $p(a,b,c) = a \cdot b^{-1} \cdot c$.
    When the variety $\catC$ is a Mal'tsev variety, the adjunction~\eqref{A} is always an admissible adjunction (see \cite{JK94} for a proof of this
    result). An interesting aspect of the adjunction \eqref{adjqndt} is that, although the variety $\Qnd$ of quandles is not a Mal'tsev variety, one
    can still find sufficient congruences that permute, making the adjunction \eqref{adjqndt} an admissible adjunction. These congruences were
    named \emph{orbit congruences} in~\cite{BLRY10}, where they were introduced in a completely different context.
\begin{definition}
    For any subgroup $N$ of $\Inn(A)$, define the relation $\sim_N \subset A \times A$ by
    \[a \sim_N b \text{ if and only if there exists } n \in N \text{ such that } a^n = b.\]
\end{definition}
This relation is actually an equivalence relation. When moreover
$N$ is a normal subgoup of $\Inn(A)$, $\sim_N$ becomes a
congruence on $A$, called an orbit congruence (see Theorem $6.1$
in \cite{BLRY10}). One can show (Lemma $2.6$ in~\cite{EG14}) that
the orbit congruences in $\Qnd$ \emph{permute} with any other
reflexive relation $R$: $\sim_N \circ R = R \circ \sim_N$. Since
the kernel congruence $\Eq(\eta_A)$ of the $A$-component of the
unit of the adjunction~\eqref{adjqndt} $\eta_A \colon A
\rightarrow \pi_0 (A)$ is an orbit congruence (it is
$\sim_{\Inn(A)}$), the adjunction~\eqref{adjqndt} can be shown to
be admissible for Galois theory. The proof of these facts relies
on the two following results that we shall also use later on. Note
that the following Lemma is a simple modification of  Lemma~1.7 in
\cite{EG14}.

\begin{lemma}\label{specialPushouts}
    Let
    \begin{equation*}
    \begin{tikzcd}
    A \ar[r,"g"] \ar[d,"f"'] & C \ar[d,"\overline{f}"] \\
    B \ar[r,"\overline{g}"'] & D
    \end{tikzcd}
    \end{equation*}
    be a pushout of surjective homomorphisms in $\Qnd$ such that
    \[
    \Eq(f)\circ \Eq(g) = \Eq(g)\circ \Eq(f).
    \]
    Then the canonical factorization $(f,g) \colon A \rightarrow B \times_D C$ to the pullback of $\overline{f}$ and $\overline{g}$ is a surjective homomorphism.
\end{lemma}

\begin{corollary}\label{QuandleAdj} \cite{EG14}
    For any surjective homomorphism $f \colon A \rightarrow B$ in $\Qnd$ the commutative square
    \begin{equation}\label{canonical}
    \begin{tikzcd}
    A \ar[r,"\eta_A"] \ar[d,"f"'] & \pi_0(A) \ar[d,"\pi_0(f)"] \\
    B \ar[r,"\eta_B"'] & \pi_0(B)
    \end{tikzcd}
    \end{equation}
    where $\eta$ is the unit of the adjunction \eqref{adjqndt}
    has the property that the canonical arrow~$\langle f, \eta_A \rangle \colon A \rightarrow  B {\times}_{\pi_0 (B)} \pi_0(A)$ to the pullback
    (of $\pi_0 (f)$ and $\eta_B$) is surjective.
\end{corollary}

Since our adjunction~\eqref{adjqndt} is admissible, there exists a
left adjoint to the inclusion functor $\TExt(B) \hookrightarrow
\Ext(B)$ (see  \cite{JK94}). Actually, given a surjective quandle
homomorphism $f \colon A \to B$, we can find a congruence $R_t$ on
$A$ such that $t \colon A/R_t \to B$ in the factorization
 \[
 \begin{tikzcd}
 A \ar[rr,"f"] \ar[rd] & & B\\
 & A/R_t \ar[ur,dotted,"t"'] &
 \end{tikzcd}
\]
is a universal trivial extension. The trivial extension $t$ is the
pullback of $\pi_0(f)$ along $\eta_B$ and $R_t=\Eq(f) \cap
\sim_{\Inn(A)}$ is the kernel pair of the comparison morphism
$\langle f, \eta_A \rangle \colon A \to B \times_{\pi_0(B)}
\pi_0(A)$ so that $A \to A/R_t = A \to B \times_{\pi_0(B)}
\pi_0(A)$ (see \cite{EG14}).


\section{Centralizing an extension}\label{Centralizing an extension}


This section is devoted to the proof of our main result, namely
the construction of the centralization of a quandle extension. We
first describe a general construction by using categorical
arguments, then we give a concrete and algebraic  description of
the centralization.

The first construction we consider is based on the method used by
T.~Everaert~\cite{Eve14} to centralize extensions in the case of
an adjunction~\eqref{A} where the variety $\catC$ is a Mal'tsev
variety. We will show that the same method still works for the
adjunction~\eqref{adjqndt} by pointing out a weaker condition than
the permutability of all congruences.

    \begin{lemma}\label{341}
        Consider the following pullback
        \begin{equation}\label{3511}
        \begin{tikzcd}
        E \times_B A \ar[r,"p_2"] \ar[d,"p_1"'] & A \ar[d,"f"]\\
        E \ar[r,"p"'] & B
        \end{tikzcd}
        \end{equation}
        where $f \colon A \to B$ is a surjective quandle homomorphism and $p \colon E \to B$ is a normal extension.
        Then \[\Eq(p_2) \circ (\Eq(p_1) \cap \sim_{\Inn(E \times_B A)}) = (\Eq(p_1) \cap \sim_{\Inn(E \times_B A)}) \circ \Eq(p_2). \]
    \end{lemma}

    \begin{proof}
        If $((e,a),(e',a')) \in \Eq(p_2) \circ (\Eq(p_1) \cap \sim_{\Inn(E \times_B A)})$ then there exists
        $(\epsilon, \alpha) \in E \times_B A$ such that \[(e,a) (\Eq(p_1) \cap \sim_{\Inn(E \times_B A)} ) (\epsilon,\alpha) \Eq(p_2) (e',a').\]
        Thus we have $ (\epsilon,\alpha) = (e,a')$ with $p(e) = f(a')$ and
        \[\pi_0(e,a') = \pi_0(e,a).\] The first condition implies that $p(e) = f(a') = p(e')$ while the second
        condition implies the existence of elements $(e_i,a_i) \in E \times_B A$ and $\alpha_i \in \{-1,1\}$ with $1 \leq i \leq n$
        such that \[(e,a) \lhd^{\alpha_1} (e_1,a_1) \lhd^{\alpha_2} (e_2,a_2) \dots \lhd^{\alpha_n} (e_n,a_n) = (e,a').\] In particular,
        this shows that \[p(e) = p(e')\] and \[e \lhd^{\alpha_1} e_1 \lhd^{\alpha_2} e_2 \dots \lhd^{\alpha_n} e_n = e.\] Since
        $p \colon E \to B$ is a normal extension, Proposition~\ref{normal} implies that
        \[e' \lhd^{\alpha_1} e_1 \lhd^{\alpha_2} e_2 \dots \lhd^{\alpha_n} e_n = e'.\]
        Now we have that \[(e,a) \Eq(p_2) (e',a)\] and \[(e',a) \lhd^{\alpha_1} (e_1,a_1) \lhd ^{\alpha_2} (e_2,a_2) \dots \lhd^{\alpha_n} (e_n,a_n) = (e',a')\]
        which implies that \[(e',a) (\Eq(p_1) \cap \sim_{\Inn(E \times_B A)}) (e',a')\] and thus
        \[((e,a),(e',a')) \in (\Eq(p_1) \cap \sim_{\Inn(E \times_B A)}) \circ \Eq(p_2). \]
    \end{proof}

        Note that this result does not remain true if $p$ is not asked to be a normal extension.

        \begin{example}
            Consider the involutive quandle $A$ given by the following table :
            \[\begin{tabular}{c|cccc}
            $\lhd$ & $a$ & $b$ & $c$ & $d$ \\
            \hline
            $a$ & $a$ & $a$ & $a$ & $b$ \\
            $b$ & $b$ & $b$ & $b$ & $a$ \\
            $c$ & $c$ & $c$ & $c$ & $c$\\
            $d$ & $d$ & $d$ & $d$ & $d$
            \end{tabular}\] and $B=\{x,y\}$ the two-elements trivial quandle. Now define a quandle homomorphism
            $f \colon A \to B$ by $f(a)=f(b)=f(c)=x$ and $f(d)=y$. This quandle homomorphism has a \emph{section} $s \colon B \to A$
            (this means that $f \circ s = \id_B$) defined by $s(x)=c$ and $s(y)=d$.
            Let us consider the kernel congruence of $f$, which is given by the set $\Delta \cup
            \{(a,b),(a,c),(b,a),(b,c),(c,a),(c,b)\}$, where
            $\Delta$ is the diagonal of $A$, representing the
            reflexivity part of the relation: $\Delta = \{  (x, x)
            \, | \, x \in A \}$.

            If we compute the intersection $\Eq(p_1) \cap \sim_{\Inn(A \times_B A)}$, we get \[\Delta \cup \{((c,a),(c,b)),((c,b),(c,a))\}.\]
            The idea is that members of $\Eq(p_1)$ share the same first element but the only couple that acts non trivially via the operation
            $\lhd$ is $(d,d)$ which changes the first entry of an element $(s,t)$ whenever $s = a$ or $b$. This implies that the only remaining
            elements are the elements from the diagonal and elements with ``$c$'' in the first entry.

            Now let us consider the element $((b,a),(c,b))$. This element is in $$(\Eq(p_1)\cap \sim_{\Inn(A \times_B A)}) \circ \Eq(p_2)$$ since
            \[(b,a) \Eq(p_2) (c,a) (\Eq(p_1)\cap \sim_{\Inn(A \times_B A)}) (c,b).\] But $((b,a),(c,b))$  doesn't belong to
            $\Eq(p_2) \circ (\Eq(p_1)\cap \sim_{\Inn(A \times_B A)})$ since $$(b,a) (\Eq(p_1)\cap \sim_{\Inn(A \times_B A)}) (b,a)$$ is the only
            choice in $(\Eq(p_1)\cap \sim_{\Inn(A \times_B A)})$.
        \end{example}

        This permutability property is crucial for the construction, since it allows us to show that the image of a congruence remains a congruence.

    \begin{proposition}\label{imageequivalence}
        Let $f \colon A \to B$ be a surjective quandle homomorphism and $R$ a congruence on $A$. Then $f(R)$ is a congruence
        when $R \circ \Eq(f) = \Eq(f) \circ R$.
    \end{proposition}

    \begin{proof}
        It is easy to see in general that the image of a congruence is still a reflexive and symmetric relation compatible with the operations.

Suppose that $R \circ \Eq(f) = \Eq(f) \circ R$. We have to show
that $f(R)$ is a transitive relation. For this, let $(b_1,b_2)$
and $(b_2,b_3)$ be elements of $f(R)$, then there exist $(a_1,a_2)
\in R$ such that $f(a_1)=b_1$ and $f(a_2)=b_2$ and $(a_2',a_3) \in
R$ such that $f  (a_2') = b_2$ and $f(a_3)=b_3$. In particular, we
see that $(a_2,a_2') \in \Eq(f)$. It follows that $(a_2,a_3) \in R
\circ \Eq(f)$ (since $a_2 \Eq(f) a_2' Ra_3$) but then $(a_2,a_3)
\in \Eq(f) \circ R$ which means that there exists $z \in A$ such
that $a_2 R z \Eq(f) a_3$. Remark that this implies that $f(z) =
f(a_3)=b_3$ and thus, since $R$ is transitive, $(a_1,z) \in R$,
showing that $(f(a_1),f(z)) =(b_1,b_3) \in f(R)$.
    \end{proof}

        \begin{corollary}\label{343}
            Consider the pullback~\eqref{3511} where $f \colon A \to B$ is a surjective quandle homomorphism and $p \colon E \to B$ is a normal
            extension. Then the relation $p_2(\Eq(p_1)\cap \sim_{\Inn(E \times_B A)})$ is a congruence on the quandle $A$.
        \end{corollary}

        \begin{proof}
            This follows directly from Lemma~\ref{341} and Proposition~\ref{imageequivalence}.
        \end{proof}

We are now in the position to prove our main theorem.

\begin{theorem}
            The category $\CExt(B)$ of central extensions of $B$ is a (regular epi)-reflective subcategory of the category $\Ext(B)$.
        \end{theorem}

        \begin{proof}
            Let $f \colon A \to B$ be a surjective quandle homomorphism and consider the weakly universal central extension
            $p \colon \tilde{B} \to B$, whose existence has been recalled in Section \ref{section central and normal extensions}. Take the
            pullback of $f$ along $p$:
            \begin{equation*}
            \begin{tikzcd}
            \tilde{B} \times_B A \ar[r,"p_2"] \ar[d,"p_1"'] & A \ar[d,"f"] \\
            \tilde{B} \ar[r,"p"'] & B.
            \end{tikzcd}
            \end{equation*}
            In order to simplify the notation, let us write \[\cap = \Eq(p_1)\cap \sim_{\Inn(\tilde{B} \times_B A)}. \] The congruence
            $\cap$ is actually the congruence that trivializes $p_1$, i.e. such that $l \colon (\tilde{B} \times_B A)/\cap \to \tilde{B}$
            is the universal trivial extension associated with $p_1$.
            Now, since $p \colon \tilde{B} \to B$ is a normal extension, we already
            know from Corollary~\ref{343} that $p_2(\cap)$ is a congruence on
            the quandle $A$. Now consider the following diagram
            \begin{equation*}
            \begin{tikzcd}[column sep=small]
            \cap \ar[d,shift left=0.7ex,"t_2"] \ar[d,shift right=0.7ex,"t_1"'] \ar[rr,"\phi"] & & p_2(\cap) \ar[d,shift left=0.7ex,"q_2"]
            \ar[d,shift right=0.7ex,"q_1"'] & \\
            \tilde{B} \times_B A \ar[rr,"p_2"] \ar[dd,"p_1"'] \ar[rd,"t"] & & A \ar[dd,"f"] \ar[rd,"q"]  & \\
            &  (\tilde{B} \times_B A)/\cap \ar[ld,"l"] & & A /p_2(\cap) \\
            \tilde{B} \ar[rr,"p"'] & & B.
            \end{tikzcd}
            \end{equation*}
            We first observe that since
            \begin{align*}
            q \circ p_2 \circ t_1 &= q \circ q_1 \circ \phi\\
            &= q \circ q_2 \circ \phi \\
            &= q \circ p_2 \circ t_2
            \end{align*}
            there exists a unique quandle homomorphism $h \colon (\tilde{B} \times_B A)/\cap \to A/ p_2(\cap)$ such that $h \circ t = q \circ p_2$.

            Also, since $\phi \colon \cap \to p_2(\cap)$ is a surjective quandle homomorphism and \linebreak $f \circ q_1 \circ \phi = f \circ q_2 \circ \phi$,
            we have a unique quandle homomorphism $c_f \colon A/p_2(\cap) \to B$ such that $c_f \circ q = f$.

            Now we observe that the square of surjective homomorphisms
            \begin{equation}\label{squarepushout}
            \begin{tikzcd}
            \tilde{B} \times_B A \ar[r,"p_2"] \ar[d,"t"'] & A \ar[d,"q"]\\
            (\tilde{B} \times_B A)/\cap \ar[r,"h"'] & A/p_2(\cap)
            \end{tikzcd}
            \end{equation}
            is a pushout since $\phi \colon \cap \to p_2(\cap)$, $t \colon \tilde{B} \times_B A \to (\tilde{B} \times_B A)/\cap$ and
            $q \colon A \to A/P_2(\cap)$ are surjective quandle homomorphisms. By Lemma~\ref{341} and Lemma~\ref{specialPushouts}, this
            implies that the comparison morphism $\langle t,p_2 \rangle\colon \tilde{B} \times_B A\to  (\tilde{B} \times_B A)/\cap \times_{A/p_2(\cap)} A$
            is surjective. $(p_1,p_2)$ being a jointly monomorphic pair, the pair $(t,p_2)$ is also jointly monomorphic and $\langle t,p_2\rangle$ is
            injective. Consequently, the square~\eqref{squarepushout} is a pullback.

            Now we have the following situation
            \begin{equation*}
            \begin{tikzcd}
            \tilde{B} \times_B A \ar[rd,"t"] \ar[rr,"p_2"] \ar[dd,"p_1"'] & & A \ar[dd,near start,"f"] \ar[rd,"q"] & \\
            & (\tilde{B} \times_B A)/\cap \ar[rr,near start,crossing over,"h"'] \ar[ld,"l"] & & A/p_2(\cap) \ar[ld,"c_f"] \\
            \tilde{B} \ar[rr,"p"'] & & B &
            \end{tikzcd}
            \end{equation*}
            where the back square and the top square are pullbacks. Since $q \colon A \to A/p_2(\cap)$ is a surjective quandle homomorphism, the square
            \[\begin{tikzcd}
            (\tilde{B} \times_B A)/\cap \ar[r,"h"] \ar[d,"l"'] & A/p_2(\cap) \ar[d,"c_f"] \\
            \tilde{B} \ar[r,"p"'] & B
            \end{tikzcd}\] is also a pullback \cite[Proposition 2.7]{JK94}. Since $l$ is a trivial extension, this means that
            $c_f \colon A/p_2(\cap) \to B$ is a central extension.

            All is left to show now is the universality of the construction. For this, consider the following factorization
            \begin{equation*}
            \begin{tikzcd}
            A \ar[rr,"f"] \ar[rd,"u"'] & & B \\
            & C \ar[ur,"c"'] &
            \end{tikzcd}
            \end{equation*}
            where $c \colon C \to B$ is a central extension. Take the pullback of $c$ along $p$
            \begin{equation*}
            \begin{tikzcd}
            \tilde{B} \times_B C \ar[r,"s_2"] \ar[d,"s_1"'] & C \ar[d,"c"] \\
            \tilde{B} \ar[r,"p"'] & B
            \end{tikzcd}
            \end{equation*}
            where $s_1 \colon \tilde{B} \times_B C \to \tilde{B}$ is then a trivial extension.
            Since the diagram
            \begin{equation*}
            \begin{tikzcd}
            \tilde{B} \times_B A \ar[r,"u \circ p_2"] \ar[d,"p_1"'] & C \ar[d,"c"] \\
            \tilde{B} \ar[r,"p"'] & B
            \end{tikzcd}
            \end{equation*} commutes, there is a unique quandle homomorphism $\beta \colon \tilde{B} \times_B A \to \tilde{B} \times_B C$
            such that $u \circ p_2 = s_2 \circ \beta$ and $s_1 \circ \beta = p_1$. By universality of the factorization $l \circ t$, there
            exists a unique quandle homomorphism $\gamma \colon (\tilde{B} \times_B A)/\cap \to \tilde{B} \times_B C$ such that
            $\beta = \gamma \circ t$ and $l = s_1 \circ \gamma$. Thus $s_2 \circ \gamma \circ t = s_2 \circ \beta = u \circ
            p_2$, and
            since the square~\eqref{squarepushout} is a pushout, this yields a unique quandle homomorphism $\alpha \colon A/p_2(\cap) \to C$
             such that $\alpha \circ h = s_2 \circ \gamma$ and $\alpha \circ q = u $. The latter equality implies that
            \[c \circ \alpha \circ q = c \circ u = f = c_f \circ q, \]
            and since $q$ is a surjective homomorphism, we have $c_f = c \circ \alpha$.
            \end{proof}

Let us give now a concrete, algebraic description of the
congruence needed to produce the centralization of $f$, that is, a
description of $p_2(\cap)$. We define explicitly a congruence
$R_c$ such that $R_c \subseteq \Eq(f)$ and show directly that the
induced extension $c$ in
\[
 \begin{tikzcd}
 A \ar[rr,"f"] \ar[rd,"q_{R_c}"'] & & B\\
 & A/R_c \ar[ur,dotted,"c"'] &
 \end{tikzcd}
\]
is the reflection of $f$ in $\CExt(B)$, so that $p_2(\cap) = R_c$.
Thus, this also offers another approach to centralization. We
define a relation $R$ on $A$ by putting
    \[
    R=\{(z \lhd x, z \lhd x')  \,|\, z,x,x' \in A \text{ and }f(x)=f(x') \} \subseteq A \times A.
    \]
    This relation is reflexive and symmetric but not transitive. Moreover, it is not stable under the quandle operations, so it is not
    a congruence in general. We construct $R_c$ as the congruence generated by $R$ and we denote the corresponding quotient by $q_{R_c} \colon A \to A/R_c$.

    The congruence $R_c$ is included in the kernel congruence $\Eq(f)$ of $f$. Indeed, every element of $R$ is included in the kernel congruence
    of $f$: if $(z \lhd x, z \lhd x') \in R$, then \[f(z \lhd x)=f(z) \lhd f(x) = f(z) \lhd f(x')=f(z \lhd x').\]

    The universal property of the quotient induces a unique quandle homomorphism $c \colon A/R_c \to B$, which is actually a central extension:
    if $c([a]_{R_c}) = c([a']_{R_c})$ then in particular $f(a)=f(a')$; this implies that $(z \lhd a) R_c (z \lhd a')$ and consequently
    $[z \lhd a]_{R_c}=[z \lhd a]_{R_c}$.

    Now let us see that it is universal: given another factorization $f = c' \circ q$, with $c' \colon A' \to B$ a central extension, we have
    to verify the existence of a unique $\phi \colon A/R_c \to A'$ such that $c' \circ \phi = c$ and $\phi \circ q_{R_c} = p$. Let
    $(z \lhd a, z \lhd a') \in R$, then $f(a)=f(a')$ or equivalently $c' \circ q(a)= c' \circ q(a')$. The homomorphism $c'$ being a central
    extension, we have that $w \lhd q(a) = w \lhd q(a')$ for all $w \in A'$. Thus by taking $w=q(z)$ we get $q(z \lhd a) = q(z \lhd a')$, and by
    the universal property of the quotient $q_{R_c}$ we get the desired factorization.


\section{Normalizing an extension}\label{Normalizing an extension}


Now we focus our attention to the case of normal extensions. We
are going to prove the following theorem.

\begin{theorem}\label{TheoNormalization}
The category $\NExt(\Qnd)$ of normal extensions is a (regular
epi)-reflective subcategory of the category $\Ext(\Qnd)$.
\end{theorem}

For this, we first prove that the inclusion functor $$H_1 \colon
\NExt(\Qnd) \hookrightarrow \Ext(\Qnd)$$ has a left adjoint using
the Freyd Adjoint Functor Theorem (see, for example, Theorem $9.9$
in \cite{BarrWells}):

\begin{theorem}  Given a small-complete category $\catA$ with small hom-sets, a functor $G\colon \catA\to \catX$ has a left adjoint
if and only if it preserves all small limits and satisfies the
so-called \emph{Solution Set Condition}: for each object $X \in
\catX$ there is a set $\mathcal{S}_X$ of objects of the comma
category $(X \downarrow G)$ such that for every object $Y$ of $(X
\downarrow G)$ there is a morphism $S \to Y$ in $(X \downarrow G)$
with $S$ in $\mathcal{S}_X$.
\end{theorem}

Recall that an object of $(X \downarrow G)$ is an arrow in $\catX$
of the form $f \colon X \to G(A)$ for some $A$, and a morphism in
$(X \downarrow G)$ from $f \colon X \to G(A)$ to $f' \colon X \to
G(A')$ is an arrow $a\colon A \to A'$ in $\catA$ such that $f'
=G(a) \circ f$.

We split the proof into several lemmas.

\begin{lemma} The category $\Ext(\Qnd)$ is small-complete.
\end{lemma}

\begin{proof}
Let us first recall that all small limits exist if small products
and equalizers exist. First, all small products in $\Ext(\Qnd)$
exist and are computed as in $\Qnd^\rightarrow$, i.e.~ the product
of a family $(f_i\colon A_i \to B_i)_{i\in I}$ of surjective
quandle homomorphisms is given by $\prod_{i\in I} f_i \colon
\prod_{i\in I} A_i \to \prod_{i\in I} B_i$. Now, let
$\alpha=(\alpha_1,\alpha_0) \colon f \to g$ and
$\beta=(\beta_1,\beta_0) \colon f \to g$ be two parallel morphisms
in $\Ext(\Qnd)$. The equalizer of $(\alpha,\beta)$ is given by
$((e_1,e_0\circ m) \colon p \to f)$ where $((e_1,e_0)\colon e\to
f)$ is the equaliser of $(\alpha,\beta)$ in $\Qnd^\rightarrow$ and
$p$ and $m$ come from the regular epi-mono factorization $e=m\circ
p$:
\[
\begin{tikzcd}
E_1 \ar[r,"e_1", rightarrowtail] \ar[d,"p"] \ar[dd,bend
right,"e"'] & A \ar[r,shift left=1ex, "\alpha_1"]
\ar[r,shift right=1ex,"\beta_1"'] \ar[dd,"f"] & C \ar[dd,"g"]\\
I \ar[d,"m"] \ar[rd,"e_0\circ m", dotted, near start] & & \\
E_0 \ar[r,"e_0"'] & B \ar[r, "\alpha_0", shift left=1ex]
\ar[r,"\beta_0"', shift right=1ex] & D
\end{tikzcd}
\]
\end{proof}

\begin{lemma} $\NExt(\Qnd)$ is closed under subobjects in $\Ext(\Qnd)$.
\end{lemma}
\begin{proof}
A morphism $(\alpha_1, \alpha_0)$ in $\Ext(\Qnd)$
\[
\begin{tikzcd}
A \ar[r,"\alpha_1"] \ar[d,"f"'] & C \ar[d,"g"] \\
B \ar[r,"\alpha_0"'] & D
\end{tikzcd}
\]
is a monomorphism if and only if $\alpha_1$  is an injective
homomorphism. Indeed, suppose that $\alpha_1$ is injective and
consider the diagram:
\[\begin{tikzcd}
Z \ar[r,shift left, "m_1"] \ar[r,shift right, "m_2"'] \ar[d,"h"'] & A \ar[r,"\alpha_1"] \ar[d,"f"'] & C \ar[d,"g"]\\
X \ar[r,shift left, "n_1"] \ar[r,shift right, "n_2"'] & B
\ar[r,"\alpha_0"] & D
\end{tikzcd}\]
where $\alpha_1 \circ m_1 = \alpha_1 \circ  m_2$ and $\alpha_0
\circ n_1 = \alpha_0 \circ n_2$.  Since  $\alpha_1$ is injective,
we get that $m_1  = m_2$ and from the surjectivity of $h$ we
obtain that $n_1 = n_2$. Conversely, suppose that $(\alpha_1,
\alpha_0)$ is a monomorphism in $\Ext(\Qnd)$ and consider two
quandle  homomorphisms $m_1, m_2 \colon Z \to A$ such that
$\alpha_1 \circ m_1 = \alpha_1 \circ  m_2$. Then we can consider
the following commutative diagram:
\[\begin{tikzcd}
Z \ar[r,shift left, "m_1"] \ar[r,shift right, "m_2"'] \ar[d,equal] & A \ar[r,"\alpha_1"] \ar[d,"f"'] & C \ar[d,"g"]\\
X \ar[r,shift left, "f\circ m_1"] \ar[r,shift right, "f\circ
m_2"'] & B \ar[r,"\alpha_0"] & D
\end{tikzcd}\]
Being $(\alpha_1, \alpha_0)$ a monomorphism in $\Ext(\Qnd)$, we
get that $m_1 = m_2$.

We want to prove that if $g$ is a normal extension and $\alpha_1$
an injective homomorphism, $f$ is a normal extension, too.

Let $a_i, a_i'$ be elements of $A$ and $\alpha_i$ be elements of
$\mathbb{Z}$ for $0\leq i \leq n$ such that
\[
a_0 \lhd^{\alpha_1} a_1 \lhd^{\alpha_2} \dots \lhd^{\alpha_n} a_n
= a_0
\]
and  $f(a_i)=f(a_i')$. We must show that
\[
a'_0 \lhd^{\alpha_1} a'_1 \lhd^{\alpha_2} \dots \lhd^{\alpha_n}
a'_n = a'_0.
\]
For this, it suffices to show that
\[
\alpha_1(a'_0) \lhd^{\alpha_1} \alpha_1(a'_1) \lhd^{\alpha_2}
\dots \lhd^{\alpha_n} \alpha_1(a'_n) = \alpha_1(a'_0)
\]
since $\alpha_1$ is injective. But we know that
\[
\alpha_1(a_0) \lhd^{\alpha_1} \alpha_1(a_1) \lhd^{\alpha_2} \dots
\lhd^{\alpha_n} \alpha_1(a_n) = \alpha_1(a_0)
\]
and
$g(\alpha_1(a_i))=\alpha_0(f(a_i))=\alpha_0(f(a'_i))=g(\alpha_1(a'_i))$.
The result follows by normality of $g$.
\end{proof}

\begin{lemma} $\NExt(\Qnd)$ is small-complete and $H_1$ preserves small limits.
\end{lemma}

\begin{proof}
It is easy to show that the subcategory $\NExt(\Qnd)$ is closed
under small products computed in $\Ext(\Qnd)$. The fact that it is
also closed under equalizers follows from the fact that it is
closed under subobjects. Consequently, $\NExt(\Qnd)$ is closed
under all small-limits in $\Ext(\Qnd)$ and the result follows.
\end{proof}

\begin{lemma} The inclusion functor $H_1$ satisfies the Solution Set
Condition.
\end{lemma}
\begin{proof}
For any quandle $A$, we define
\[
Q_A=\left\{\text{canonical projection } \pi \colon A \to A/R \, \,
|\,  \,\text{$R$ congruence on $A$} \right\}
\]
to be the set of canonical representatives of quotients of $A$
(there is only a set of congruences on $A$). As in every variety,
any surjective quandle homomorphism $q\colon A \to Q$ is
isomorphic to an element $\pi_q \colon A \to A/\Eq(q)$ of $Q_A$,
i.e.~$q=i_q \circ \pi_q$ for a (unique) isomorphism $i_q$:
\[
\begin{tikzcd}
A \ar[r,"q"] \ar[d,"\pi_q"'] & Q \ar[d,equal] \\
A/\Eq(q) \ar[r,"i_q"'] & \Im(q)
\end{tikzcd}
\]
Let $f \colon A\to B$ be an object of $\Ext(\Qnd)$. Consider the
class of pairs
\[
(q \colon A \to Q, g \colon Q \to B)
\]
with $q$ a surjective homomorphism and $g$ a normal extension such
that  $f=g\circ q$. Since $q$ is an epimorphism, such a $g$ is
unique, hence the existence of $g$ can be viewed as a property
$\varphi(q)$ of the morphism $q$:
\[
\varphi(q) \equiv \text{ there exists $g$ normal such that
$f=g\circ q$.}
\]
 It is obvious that $\varphi(q)$ holds if and only if $\varphi(\pi_q)$ holds. We define $\mathcal{S}_f$ to be the set
\[
S_f = \left\{ (\pi,1_B)\colon f \to H_1(g) \,\,|\,\, \text{$\pi$
is in $Q_A$}\right\}
\]
of objects of $(f\downarrow H_1)$ and $\overline{\mathcal{S}}_f$
to be the class of objects of $(f\downarrow H_1)$ of the form
\[
(q,1_B)\colon f \to H_1(g)
\]
where $q$ is a surjective homomorphism (and $g$ is normal). For
any object $y=(q,1_B)\colon f \to H_1(g) $ of
$\overline{\mathcal{S}}_f$, there is a morphism from the object
\[
s_y= (\pi_q,1_B)\colon f \to H_1(g\circ i_q)
\]
in $\mathcal{S}_f$ to $y$ given by
\[
(i_q,1_B)\colon g\circ i_q \to g.
\]
Now let $y=(\alpha_1,\alpha_0)\colon f \to H_1(g)$ be an object of
$(f\downarrow H_1)$. Consider the commutative diagram in $\Qnd$
\[
\begin{tikzcd}
A \ar[rrr,"\alpha_1"] \ar[rrd,"{\langle f , \alpha_1 \rangle}"] \ar[rd,"p"'] \ar[dd,"f"']& & & C \ar[dd,"g"] \\
 & I \ar[d,"g''"'] \ar[r,"m"'] & B\times_D C \ar[ur,"\alpha_0'"] \ar[d,"g'"] & \\
B \ar[r,equal] & B \ar[r,equal] & B \ar[r,"\alpha_0"] & D
\end{tikzcd}
\]
where
\begin{itemize}
    \item $(B\times_D C, g',\alpha_0')$ is a pullback of $\alpha_0$ and $g$;
    \item $\langle f, \alpha_1 \rangle$ is the unique induced morphism such that $f = g' \circ \langle f, \alpha_1 \rangle$
    and $\alpha_1=\alpha_0' \circ \langle f, \alpha_1 \rangle$;
    \item $\langle f, \alpha_1 \rangle= m\circ p$ is a (regular epi)-mono factorization of $\langle f, \alpha_1 \rangle$;
    \item $g''= g'\circ m$.
\end{itemize}
Since $f$ is surjective, $g''$ is a surjective homomorphism, too.
Therefore, it is also a subobject in $\Ext(\Qnd)$ of the normal
extension $g'$ and, consequently, a normal extension. We define
$y'$ to be the object of $\overline{\mathcal{S}}_f$
\[
y'= (p,1_B)\colon f \to H_1(g'')
\]
and we see that there is an arrow from $y'$ to $y$, namely
\[
(\alpha_0'\circ m,\alpha_0)\colon g'' \to g.
\]
Combining this with the previous observations, we find at least
one object $s=s_{y'}$ in the set $\mathcal{S}_f$ and an arrow from
$s$ to $y$. This  concludes the proof.
\end{proof}

We have just proved that a left adjoint to $H_1$ exists. It
remains to prove that every component of the unit of the
adjunction is a regular epimorphism. This comes from the following
observations:
    \begin{enumerate}
        \item It is easy to show that, for an object $s$ in $S_f$ (or $\overline{S}_f$) and $y$ in $(f\downarrow H_1)$, there is at most one
        arrow $s\to y$ in the category $(f\downarrow H_1)$.
        \item

        Consequently, the initial object
        of $(f\downarrow H_1)$ (i.e. the reflection of $f$ in $\NExt(\Qnd)$) can be chosen in $S_f$. Indeed, it is direct to check that, for
        $\overline{y}$ the reflection of $f$ in $\NExt(\Qnd)$, $s_{\overline{y}'}$ is also initial, so that $s_{\overline{y}'}\cong \overline{y}$.
        This means that the reflection of $f$ in $\NExt(\Qnd)$ is, up to isomorphisms, of the form $\eta_f^1 = (\pi,1_B)\colon f \to H_1(n_f)$ for
        some $\pi$ in $Q_A$ (we denote by $n_f$ the normalization of $f$). If
        we write $(R_n,\pi_1,\pi_2)$ for the congruence associated with $\pi$,  we find that
\[
\xymatrix{
R_n \ar@<0.5ex>[r]^-{\pi_1} \ar@<-0.5ex>[r]_-{\pi_2} \ar[d]_-{\pi_1\circ f= \pi_2\circ f} & A \ar[d]^-f \ar[r]^-{\pi} & A/R \ar[d]^-{n_f} \\
B \ar@<0.5ex>@{=}[r] \ar@<-0.5ex>@{=}[r]& B \ar@{=}[r]& B }
\]
is a coequalizer diagram in $\Ext(\Qnd)$.  This conclude the proof
of Theorem~\ref{TheoNormalization}.
\end{enumerate}

 Note that we can prove in a very similar way that the inclusion $\CExt(\Qnd) \hookrightarrow \Ext(\Qnd)$ has a left adjoint. Indeed, the
 category $\CExt(\Qnd)$ enjoy the same closure properties as $\NExt(\Qnd)$ inside $\Ext(\Qnd)$.

More generally, the same kind of proof (with some obvious
modifications) can be used to obtain the following general result:

\begin{theorem}
Let
\[
\begin{tikzcd}
\catC \arrow[rr,bend left, "I"] & \perp & \catH \arrow[ll,bend
left, "\supseteq"]
\end{tikzcd}
\]
be an admissible adjunction, where $\catC$ is a variety and
$\catH$ is a subvariety of $\catC$. Let us assume that
$\NExt(\catC)$ (resp.~$\CExt(\catC)$) is closed under subobjects
and small products in $\Ext(\catC)$. Then $\NExt(\catC)$
(resp.~$\CExt(\catC)$) is a (regular epi)-reflective subcategory
of $\Ext(\catC)$.
\end{theorem}

Contrary to the case of central extensions, for normal extensions
we haven't been able to find a concrete description of the
congruence $R_n$ whose quotient gives the normalization of a
quandle extension, since the Adjoint Functor Theorem is not really
constructive. A description of such a congruence (whose existence
is a consequence of Theorem \ref{TheoNormalization}) is an open
problem.


\section*{Acknowledgements}


This work was partially supported by the FCT - Funda\c{c}\~ao para
a Ci\^encia e a Tecnologia - under the grant number
SFRH/BPD/98155/2013, and by the Centre for Mathematics of the
University of Coimbra -- UID/MAT/00324/2013, funded by the
Portuguese Government through FCT/MEC and co-funded by the
European Regional Development Fund through the Partnership
Agreement PT2020.

This work was partially supported by a FNRS grant
\textit{Cr\'{e}dit pour bref s\'{e}jour \`{a} l'\'{e}tranger} that
has allowed the second author to visit the Universidade de
Coimbra.

This work was partially supported by the Programma per Giovani
Ricercatori ``Rita Levi-Montalcini'', funded by the Italian
government through MIUR.


\end{document}